\documentclass{amsart}

\usepackage{amssymb}
\usepackage{enumerate}
\usepackage{tikz-cd}
\usepackage{stmaryrd}

\theoremstyle{plain}
\newtheorem{thm}{Theorem}[section]
\newtheorem{lem}[thm]{Lemma}
\newtheorem{prop}[thm]{Proposition}

\theoremstyle{definition}
\newtheorem{defn}{Definition}[section]

\let\hom\relax
\DeclareMathOperator{\hom}{Hom}

\usepackage{nomencl}
\makenomenclature

\newcommand{\A}{\ensuremath{\text{Mack}_R(G)}}
\newcommand{\B}{\ensuremath{\mu_R(G)\text{-mod}}}

\title{Box Product of Mackey Functors\\ in Terms of Modules}
\author{Zhulin Li}
\date{\today}

\begin{document}

\begin{abstract}
The box product of Mackey functors has been studied extensively in Lewis's notes.
As shown in Thevenaz and Webb's paper, a Mackey functor may be identified with a module over a certain algebra, called the Mackey algebra.
We aim at describing the box product, in the sense of Mackey algebra modules.
For a cyclic $p$-group $G$, we recover a result from Mazur's thesis.
We generalize it to a general finite group $G$ in this article.
\end{abstract}

\maketitle
\section*{Introduction}
A Mackey functor is an algebraic structure, related to many natural constructions from finite groups, such as group cohomology and the algebraic K-theory of group rings.
The study of Mackey functor in abstract began in 1980s.
Dress\cite{Dress} and Green\cite{Green} first gave the axiomatic formulation of Mackey functors.
Several equivalent descriptions of Mackey functors were given by Dress\cite{Dress}, Lindner\cite{Lindner}, Lewis\cite{Lewis} and Thevenaz\cite{Thevenaz}.
Specifically, Lewis\cite{Lewis} introduces box product and Thevenaz\cite{Thevenaz} describes a Mackey functor as a module over the Mackey algebra $\mu_R(G)$.

In this article, we give an inductive description of the box product of two left $\mu_R(G)$-modules.
The main goal is to construct this box product explicitly and to prove that it is equivalent with the box product of two Mackey functors.
When $G$ is a cyclic $p$-group, we recover the formula of Mazur\cite{Mazur}.

\section*{Acknowledgement}
This work was made possible in part by support from 2015 Spring MIT Undergraduate Research Opportunities Program (UROP).
I would like to thank my supervisor Dr. Emanuele Dotto at MIT for proposing this interesting topic and for his comments that greatly improved the manuscript.
I also wish to thank Dr. Zhiyi Huang at HKU for helpful discussions, which provide lots of insight.

\tableofcontents
\nomenclature{$G$}{a finite group}
\nomenclature{$R$}{a unital and commmutative ring}
\nomenclature{$\mu_R(G)$}{the Mackey algebra}
\nomenclature{$e$}{the trivial subgroup of $G$ with only one element}
\nomenclature{$\pi^H_K$}{a map from $G/K$ to $G/H$, mapping $gK$ to $gH$}
\nomenclature{$c_{g,H}$}{a map from $G/H$ to $G/^gH$, mapping $kH$ to $(kg^{-1})^gH$}
\nomenclature{$^gH$}{$gHg^{-1}$}
\nomenclature{$H^g$}{$g^{-1}Hg$}
\nomenclature{$K'< H$}{$K'$ is a proper subgroup of $H$}
\nomenclature{$K\leq H$}{$K$ is a subgroup of $H$}
\nomenclature{$\omega(G)$}{an intermediate category of definition for Mackey functors}
\nomenclature{$\Omega_R(G)$}{a category of definition for Mackey functors}
\nomenclature{$B^G$}{the Burnside ring Mackey functor}
\nomenclature{$r_f$}{restriction of $f$}
\nomenclature{$t_f$}{transfer of $f$}
\nomenclature{\A}{the category of Mackey functors}
\nomenclature{\B}{the category of left $\mu_R(G)$-modules}
\nomenclature{$C_{g,H}$}{$G/^gH\leftarrow G/H=G/H$ or its value under some Mackey functor}
\nomenclature{$R_K^H$}{$G/K=G/K\to G/H$ or its value under some Mackey functor}
\nomenclature{$I^H_K$}{$G/H\leftarrow G/K=G/K$ or its value under some Mackey functor}
\nomenclature{$[J\backslash H/K]$}{the set of representatives of $J\backslash H/K$}
\printnomenclature

\section{Mackey functors}
Throughout this article, we assume that $G$ is a finite group and that $R$ is a unital, commutative ring.
There are several equivalent definitions of Mackey functors and we concentrate on two of them here.
Before giving the definitions, we introduce an auxiliary category $\Omega_R(G)$ and the Mackey algebra $\mu_R(G)$.
\subsection{The category $\Omega_R(G)$}
We recall the definition of the category $\Omega_R(G)$ from \cite{Lewis}.
A finite group $G$ gives rise to a category $\omega(G)$ whose objects are finite $G$-sets and where the morphisms from $X$ to $ Y$ are the equivalence classes of diagrams of $G$-sets $X\leftarrow V\to Y$. Two such diagrams are said to be equivalent if there is a commutative diagram
\begin{center}
\begin{tikzcd}
 & V \arrow{dd}[swap]{\sigma}\arrow{ld}\arrow{rd} &\\
 X & & Y\\
 & V'\arrow{lu}\arrow{ru} &
\end{tikzcd}
\end{center}
where $\sigma$ is an isomorphism of $G$-sets.
To define the composition of morphisms, we consider a morphism from $X$ to $Y$ represented by a diagram $X\leftarrow V\to Y$ and a morphism from $Y$ to $Z$ represented by a diagram $Y\leftarrow W\to Z$.
We form the pullback
\begin{center}
\begin{tikzcd}
&&U\arrow{ld}{f'}\arrow{rd}{g'}&&\\
& V\arrow{rd}{f}\arrow{ld}&&W\arrow{ld}{g}\arrow{rd}&\\
X&&Y&&Z
\end{tikzcd}
\end{center}
which defines a diagram $X\leftarrow U\to Z$, hence a morphism from $X$ to $Z$.
Such a pullback always exists and we can express it explicitly as \[U = \{(v, w)\in V\times W:f(v) = g(w)\},\] where $f'(v, w) = v$ and $g'(v, w) = w$.
By defining addition in $\hom_{\omega(G)}(X,Y)$ as
$$(X\xleftarrow{\alpha}V\xrightarrow{\beta}Y) + (X\xleftarrow{\alpha'}V'\xrightarrow{\beta'}Y) := (X\xleftarrow{\alpha+\alpha'}V\sqcup V'\xrightarrow{\beta+\beta'}Y),$$
$\hom_{\omega(G)(X,Y)}$ becomes a free abelian monoid, as shown in \cite{Thevenaz}.
Extending the scalars to the ring $R$, we get a free $R$-module $$\hom_{\Omega_R(G)}(X,Y) := R\hom_{\omega(G)}(X,Y)$$ on the same basis as that of free monoid $\hom_{\omega(G)}(X,Y)$.
Let $\Omega_R(G)$ be a category with the same objects as $\omega(G)$ and hom-set $\hom_{\Omega_R(G)}(X,Y)$ defined as above.
$\Omega_R(G)$ is an $R$-additive category.

\subsection{The Mackey algebra $\mu_R(G)$}
We define the Mackey algebra as $$\mu_R(G) := \bigoplus_{H,K\leq G}\hom_{\Omega_R(G)}(G/H, G/K),$$ where the multiplication is defined on the components in the direct sum by composition of morphisms in the category $\Omega_R(G)$, or zero if two morphisms cannot be composed.

For convenience, we recall the notation from \cite{Shulman}.
\begin{defn}
Let $f:X\to Y$ be a $G$-equivariant map of finite $G$-sets.
Then two spans
\begin{center}
\begin{tikzcd}
& X\arrow{rd}{f}\arrow[-, double equal sign distance]{ld} &\\
X & & Y
\end{tikzcd}
\text{ and }
\begin{tikzcd}
& X\arrow{ld}{f}\arrow[-, double equal sign distance]{rd} &\\
Y & & X
\end{tikzcd}
\end{center}
are called the restriction $r_f$ and transfer $t_f$, respectively, of $f$.
\end{defn}
\begin{defn}
Let $K\leq H\leq G$ and $g\in G$.
Then define $R^H_K$, $I^H_K$ and $C_{g, H}$ as
\begin{equation*}\begin{split}
&R^H_K = r_{\pi^H_K}: \left(G/K=G/K\xrightarrow[]{\pi^H_K} G/H\right)\\
&I^H_K = t_{\pi^H_K}: \left(G/H\xleftarrow[]{\pi^H_K} G/K=G/K\right)\\
&C_{g, H} = t_{c_{g,H}} = r_{c_{g^{-1}, ^gH}}:\left(G/^gH\xleftarrow[]{c_{g,H}} G/H = G/H\right),
\end{split}\end{equation*}
where $\pi_K^H:G/K\to G/H$ denotes the canonical quotient map, mapping $gK$ to $gH$, and $c_{g, H}:G/H\to G/^gH$ denotes the conjugation map, mapping $kH$ to $(kg^{-1})^gH$.
Observe that $R^H_K, I^H_K$ and $C_{g, H}$ are all elements in Mackey algebra $\mu_R(G)$.
\end{defn}

For conjuation, the notation $C_g$ is preferred to $C_{g, H}$ for simplicity when there is no ambiguity.
It is easy to check that the following identities hold:
\begin{enumerate}[(1)]
\setcounter{enumi}{-1}
\item $R^H_H = I^H_H = C_{h,H}$ for all $H\leq G$ and $h\in H$
\item $R^K_J R^H_K = R^H_J$ for all subgroups $J\leq K\leq H$
\item $I^H_K I^K_J = I^H_J$ for all subgroups $J\leq K\leq H$
\item $C_g C_h = C_{gh}$ for all $g, h\in G$
\item $C_g R^H_K = R^{^gH}_{^gK}C_g$ for all subgroups $K\leq H$ and $g\in G$
\item $C_gI_K^H = I_{^gK}^{^gH}C_g$ for all subgroups $K\leq H$ and $g\in G$
\item $R^H_J I_K^H = \sum_{g\in [J\backslash H/K]} I_{^gK\cap J}^J C_g R^K_{K\cap J^g}$ for all subgroups $J,K\leq H$
\item $\sum_{H\leq G}I_H^H$ serves as the unit in $\mu_R(G)$.
\end{enumerate}

We use $[J\backslash H/K]$ to denote the set of representatives of the double coset $J\backslash H/K$.
In fact, the structure of Mackey algebra $\mu_R(G)$ is rather simple:
\begin{lem}[from \cite{Thevenaz}]
Hom-set $\hom_{\Omega_R(G)}(G/K, G/H)$ is a free $R$-module, with basis represented by the diagrams
\begin{equation*}\begin{split}
I^K_{^gL}C_{g,L}R^H_L = \left(G/K\xleftarrow[]{\pi^K_{^gL}c_{g,L}}G/L\xrightarrow[]{\pi^H_L}G/H\right)
\end{split}\end{equation*}
where $g\in [K\backslash G/H]$ and $L$ is a subgroup of $H\cap K^g$ taken up to $H\cap K^g$-conjugation.
\end{lem}
In other words, the Mackey algebra $\mu_R(G)$ is generated by $R^H_K$, $I^H_K$ and $C_{g, H}$'s as an $R$-algebra.

\subsection{Definition of Mackey functors}
\begin{defn}[from \cite{Dress}]
A Mackey functor is a $R$-additive functor $M:\Omega_R(G)^{\text{op}}\to R\text{-mod}$.
They form a category with natural transformations as morphisms and we denote this category by \A.
\end{defn}

It is shown in \cite{Thevenaz} that the category $\B$ of left $\mu_R(G)$-modules is equivalent to $\A$ via the following equivalence of categories
\begin{equation*}\begin{split}
\Phi:\A &\longleftrightarrow \B\\
M&\longmapsto\oplus_{H\leq G}M(G/H)\\
\left(G/H\mapsto I_H^H N\right)&\longmapsfrom N.
\end{split}\end{equation*}
Since $\sum_{H\leq G}I^H_H$ is the unit in $\mu_R(G)$, a left $\mu_R(G)$-module $N$ can be graded into $N=\oplus_{H\leq G}I^H_HN$.
Observe that the multiplication by $R^H_K$ maps $I^H_H N$ to $I^K_K N$ and the other grades to zero.
Similarly, multiplication by $I^H_K$ maps $I^K_KN$ to $I^H_H N$ and the other grades to zero.
Multiplication by $C_{g,H}$ maps $I^H_H N$ to $I^{^gH}_{^gH}N$ and the other grades to zero.
Also note that $I^H_H\Phi M = M(G/H)$ for later use.

\section{The Box product in \A}
The box product is a symmetric monoidal structure on \A.
The box product in $\A$ has been studied in \cite{Lewis} and we summarize it in this section.
The result we are more interested in is that maps from the box product of $M$ and $N$ to $P$ can be characterized by Dress parings.

Given two Mackey functors $M,N\in\A$, we can form the exterior product
\begin{equation*}\begin{split}
M\overline{\square}N:\Omega_R(G)^{\text{op}}\times\Omega_R(G)^{\text{op}}&\longrightarrow R\text{-mod}\\
(X, Y)&\longmapsto M(X)\otimes N(Y).
\end{split}\end{equation*}
\begin{defn}[Box product in \A]
\label{def:box_product_1}
The box product $M\square N$ is defined to be the left Kan extension of $M\overline{\square}N$ along the Cartesian product functor $\times:\Omega_R(G)^{\text{op}}\times\Omega_R(G)^{\text{op}}\longrightarrow \Omega_R(G)^{\text{op}}$.
\begin{center}
\begin{tikzcd}
\Omega_R(G)^{\text{op}}\times\Omega_R(G)^{\text{op}}\arrow{r}{M\overline{\square}N}\arrow{d}{\times} & R\text{-mod}\\
\Omega_R(G)^{\text{op}}\arrow[dashrightarrow]{ru}[swap]{M\square N}&
\end{tikzcd}
\end{center}
\end{defn}

If a Mackey functor $M\in\A$ is implicit, we use $r_f$ for both a morphism $\left(X=X\xrightarrow[]{f}Y\right)$ in $\Omega_R(G)$ and its value $M(Y)\to M(X)$ under $M$.
Similarly for $t_f$.

\begin{lem}[from \cite{Lewis}]
\label{lem:dress}
A map $\theta:M\square N\to P$ determines and is determined by a collection of $R$-module homomorphisms $$\theta_{X}:M(X)\otimes N(X)\to P(X)$$ for every finite $G$-set $X$, such that the following three diagrams commute for each $G$-equivariant map $f:X\to Y$.
\begin{center}
\begin{tikzcd}
M(Y)\otimes N(Y)\arrow{r}{\theta_Y}\arrow{d}{r_f\otimes r_f} & P(Y)\arrow{d}{r_f}\\
M(X)\otimes N(X)\arrow{r}{\theta_X} & P(X)
\end{tikzcd}\\
\begin{tikzcd}
& M(X)\otimes N(X)\arrow{r}{\theta_X} & P(X)\arrow{dd}{t_f}\\
M(X)\otimes N(Y)\arrow{ru}{id\otimes r_f}\arrow{rd}{t_f\otimes id} & &\\
& M(Y)\otimes N(Y)\arrow{r}{\theta_Y} & P(Y)
\end{tikzcd}\\
\begin{tikzcd}
& M(X)\otimes N(X)\arrow{r}{\theta_X} & P(X)\arrow{dd}{t_f}\\
M(Y)\otimes N(X)\arrow{ru}{r_f\otimes id}\arrow{rd}{id\otimes t_f} & &\\
& M(Y)\otimes N(Y)\arrow{r}{\theta_Y} & P(Y)
\end{tikzcd}
\end{center}
\end{lem}
A good exposition and proof of this lemma can be found in \cite{Shulman}.
The data in this lemma is called a Dress paring.
The natural transformations from $M\square N$ to another Mackey functor $P$ are the same as the Dress parings from $M$ and $N$ to $P$, via the natural bijection $$\hom_{\A}(M\square N, P)\cong\text{Dress}(M, N; P),$$ that maps a map $\theta:M\square N\to P$ to $\theta_X:M(X)\otimes N(X)\to P(X\times X)\xrightarrow[]{r\Delta} P(X)$.
The first map comes from the Kan adjunction and $r\Delta$ is the restriction associated to the diagonal map of $G$-sets $X\to X\times X$.

In \A, the Burnside ring Mackey functor $$B^G(-) := \hom_{\Omega_R(G)}(-, G/G)$$ is the unit for box product, as shown in \cite{Lewis}.
In this way, $(\A, B^G(-), \square)$ is a symmetric monoidal category.

\section{The Box Product in $\B$}
Given two left $\mu_R(G)$-modules $M$ an $N$, we can form an $R$-module $A_H$ for each $H\leq G$ by induction on the cardinality of subgroups of $G$
\begin{equation*}\begin{split}
A_e & := (I^e_eM)\otimes_R(I^e_e N)\\
A_H & := \left((I^H_HM)\otimes_R(I^H_HN)\right) \oplus \bigoplus_{K<H}A_K.
\end{split}\end{equation*}
We say $A_H$ is of grade $H$.
By combining all the grades together, we get $A:=\oplus_{H\leq G}A_H$.
Note that the component $A_K$ in grade $H$ is distinct from the grade $K$.
Since $\mu_R(G)$ is generated by $R,I,C$'s, we can endow $A$ with a left $\mu_R(G)$-module structure by giving actions of $R,I,C$'s on $A$.
The maps $C_{g, H}, I_H^L, R^H_J$ map the grade $H$ to grades $^gH, L, J$ respectively, and map the other grades to zero.
Their action on the grade $H$ is described as follows:
\begin{enumerate}[(1)]
\item $I_H^L$ action on the grade $H$:\\
If $H = L$, $I_H^L$ acts as the identity on grade $H$.\\
If $H < L$, $I_H^L$ maps an element in $A_H$ to its corresponding copy $A_H$ in grade $L$.
To distinguish $A_H$ from its copy in grade $L$, we write its copy in grade $L$ as $I_H^L A_H$ from now on.
That is, grade $H$ is written as $$A_H = \left((I^H_H M)\otimes_R(I^H_H N)\right) \oplus \bigoplus_{K<H}I^H_K A_K.$$
\item $C_{g, H}$ action on the grade $H$:\\
We define the action of $C_{g, H}$ by induction on the cardinality of $H$ as follows.
For $m\otimes n\in (I^H_HM)\otimes (I^H_HN)$ and $x\in A_K$, where $K<H$,
\begin{equation*}\begin{split}
C_{g,H}(m\otimes n) &:= (C_{g,H}m)\otimes(C_{g,H}n)\in A_{^gH}\\
C_{g,H}I^H_K(x) &:= I^{^gH}_{^gK}C_{g, K}(x)\in A_{^gH}.
\end{split}\end{equation*}
\item $R^H_J$ action on the grade $H$:\\
We also define the action of $R_J^H$ by induction on the cardinality of $H$ as follows.
For $m\otimes n\in (I^H_HM)\otimes (I^H_HN)$ and $x\in A_K$, where $K<H$,
\begin{equation*}\begin{split}
R^H_J(m\otimes n) &:= (R^H_Jm)\otimes (R^H_Jn)\in A_J\\
R^H_JI^H_K(x) &:= \sum_{g\in [J\backslash H/K]} I_{^gK\cap J}^J C_g R^K_{K\cap J^g}(x)\in A_J.
\end{split}\end{equation*}
\end{enumerate}
\begin{defn}[Box product in \B]
\label{def:box_product_2}
Based on the left $\mu_R(G)$-module $\oplus_{H\leq G}A_H$, we can define $M\square N$ as $$M\square N:= \left(\oplus_{H\leq G}A_H\right)/FR,$$ where $FR$ is a submodule, called the Frobenius reciprocity submodule, generated by elements of the form
$$a\otimes (I_K^Hb) - I_K^H((R^H_Ka)\otimes b)$$
and $$(I_K^Hc)\otimes d - I_K^H(c\otimes (R_K^H d))$$ for all $K<H$, $a\in I^H_HM$, $b\in I^K_KN$, $c\in I^K_KM$ and $d\in I^H_HN$.
\end{defn}
Naturally, the image of $A_H$ under the quotient is called the grade $H$ of $M\square N$.

\begin{prop}
$\oplus_{H\leq G}\hom_{\Omega_R(G)}(G/H, G/G)$ is the unit for box product in \B.
\end{prop}
\begin{proof}
Let $M$ be a left $\mu_R(G)$-module and $N = \oplus_{H\leq G}\hom_{\Omega_R(G)}(G/H, G/G)$.
Then $I^H_HN$ is an $R$-module generated by $G/H\xleftarrow[]{\pi^H_L}G/L\to G/G = I^H_LR^H_Ln_H$, where $$n_H:=\left(G/H=G/H\to G/G\right).$$
Thus, $I^H_HM\otimes I^H_HN$ is a $R$-module generated by $m\otimes I^H_LR^H_Ln_H = I^H_L(R^H_Lm\otimes n_H)$.
Then we get a natural bijection between each grade of $M$ and $M\square N$
\begin{equation*}\begin{split}
F:I^H_HM&\longleftrightarrow I^H_H(M\square N)\\
m&\longmapsto m\otimes n_H\\
I^H_LR^H_Lm&\longmapsfrom I^H_L(R^H_Lm\otimes n_H)\in I^H_HM\otimes I^H_HN\\
I^H_KF^{-1}(x)&\longmapsfrom I^H_Kx\in I^H_KI^K_K(M\square N).
\end{split}\end{equation*}
Here $F^{-1}$ is defined by induction on the grade.
It is easy to check that this is a bijection and that it preserves the $\mu_R(G)$-module structure.
Thus, $M\square N$ is isomorphic to $M$ as a $\mu_R(G)$-module and $N$ is the unit for box product in \B.
\end{proof}
In this way, $(\B, \oplus_{H\leq G}\hom_{\Omega_R(G)}(G/H, G/G), \square)$ is a symmetric monoidal category.

\section{Equivalence of box product in $\A$ and \B}
\begin{thm}
\label{thm:main}
The equivalence of categories $\Phi:\A\xrightarrow[]{\cong}\B$ is a symmetric monoidal equivalence.
In other words, there are a natural isomorphism $\Phi M\square \Phi N \cong \Phi(M\square N)$ for any two Mackey functors $M, N\in\A$, and a compatible natural isomorphism $B^G(-)\xrightarrow[]{\cong}\hom_{\Omega_R(G)}(G/H, G/G)$.
\end{thm}
\begin{lem}
\label{lem:main_lem}
For any Mackey functors $M, N, P\in\A$, there is a natural bijection $$\text{Dress}(M, N; P)\cong\hom_{\B}(\Phi M\square\Phi N, \Phi P).$$
\end{lem}
\begin{proof}
Given $\beta\in\hom_{\B}(\Phi M\square\Phi N, \Phi P)$, we map it to $\theta\in\text{Dress}(M, N; P)$ defined as follows.
For each $H\leq G$, $\beta$ maps grade $I^H_H(\Phi M\square\Phi N)$ to grade $I^H_H\Phi P$, because $\beta(x) = \beta(I^H_Hx) = I^H_H\beta(x)\in I^H_H\Phi P$ for each $x\in I^H_H(\Phi M\square \Phi N)$.
Since $$I^H_H(\Phi M\square\Phi N) = (M(G/H)\otimes N(G/H))\oplus\bigoplus_{K<H}I_K^HA_K/FR$$ and $$I^H_H\Phi P = P(G/H),$$ $\beta$ induces an $R$-module homomorphism $\theta_{G/H}$ $$\theta_{G/H}:M(G/H)\otimes N(G/H)\to P(G/H)$$ by restricting to the first summand.
For a general finite $G$-set $X = \sqcup_{i=1}^p G/H_i$, $\theta_X$ is defined as 
\begin{equation*}\begin{split}
\theta_X:\bigoplus_{i,j=1}^p M(G/H_i)\otimes N(G/H_j)&\longrightarrow\bigoplus_{i=1}^pP(G/H_i)\\
m_i\otimes n_j & \longmapsto\left\{\begin{split}
\theta_{G/H_i}(m_i\otimes n_i) & \text{, if } i = j\\
0 & \text{, otherwise.}
\end{split}\right.
\end{split}\end{equation*}

Having constructed $\theta$, we now proceed to show that $\theta\in\text{Dress}(M, N; P)$.
Given finite $G$-sets $X,Y$ and a $G$-equivariant map $f:X\to Y$, it is sufficient to show that the three diagrams in Lemma \ref{lem:dress} commute.

If both $X$ and $Y$ are orbits, say $X = G/K$ and $Y = G/H$, observe that a $G$-equivariant map $f$ from $G/K$ to $G/H$ must be of the the form $f = \pi^H_{^gK}c_{g,K}$ for some $g\in G$.
By composition of commuting diagrams, we only need to consider the case where $f = c$ and $f = \pi$.

When $f = c_{g, H}:G/H\to G/^gH$, we have that $r_f = C_{g^{-1}, ^gH}$ and $t_f = C_{g, H}$.
The first diagram commutes because 
\begin{equation*}\begin{split}
r_f\theta_{Y}(m\otimes n)
& = C_{g^{-1}, ^gH}\beta(m\otimes n) = \beta(C_{g^{-1}, ^gH}(m\otimes n))\\
& = \beta(C_{g^{-1}, ^gH}m\otimes C_{g^{-1}, ^gH}n)) = \theta_{X}(r_fm\otimes r_fn).
\end{split}\end{equation*}
The second diagram commutes because
\begin{equation*}\begin{split}
t_f\theta_{X}(m\otimes r_fn)
& = C_{g,H}\beta(m\otimes C_{g^{-1}, ^gH}n) = \beta(C_{g,H}(m\otimes C_{g^{-1}, ^gH}n))\\
& = \beta(C_{g,H}m\otimes n) = \theta_{Y}(t_fm\otimes n).
\end{split}\end{equation*}
The third diagram commutes similarly.

When $f = \pi^H_K:G/K\to G/H$, we have that $r_f = R^H_K$ and $t_f = I_K^H$.
The first diagram commutes because 
\begin{equation*}\begin{split}
r_f\theta_{Y}(m\otimes n)
& = R^H_K\beta(m\otimes n) = \beta(R^H_K(m\otimes n))\\
& = \beta(R^H_Km\otimes R^H_Kn)) = \theta_{X}(r_fm\otimes r_fn).
\end{split}\end{equation*}
The second diagram commutes because
\begin{equation*}\begin{split}
t_f\theta_{X}(m\otimes r_fn)
& = I_K^H\beta(m\otimes R^H_Kn) = \beta(I_K^H(m\otimes R^H_Kn))\\
& = \beta(I_K^Hm\otimes n) = \theta_{Y}(t_fm\otimes n).
\end{split}\end{equation*}
The third diagram commutes similarly.

Now, if $Y$ is an orbit $G/H$ and $X = \sqcup_{i=1}^p G/K_i$ is a general $G$-set, denote the restriction of $f$ to $G/K_i$ by $f_i:G/K_i\to G/H$.
Thus, $r_f:M(G/H)\to\oplus_{i=1}^p M(G/K_i)$ is the sum of $r_{f_i}:M(G/H)\to M(G/K_i)$.
Similarly, $t_f:\oplus_{i=1}^p M(G/K_i)\to M(G/H)$ is determined by components $t_{f_i}:M(G/K_i)\to M(G/H)$.
The first diagram commutes because
\begin{equation*}\begin{split}
r_f\theta_{Y}(m\otimes n)
& = \sum_{i=1}^pr_{f_i}\theta_{Y}(m\otimes n) = \sum_{i=1}^p\theta_{X}(r_{f_i}m\otimes r_{f_i}n)\\
& = \theta_{X}(\sum_{i=1}^pr_{f_i}m\otimes r_{f_i}n) = \theta_{X}(\sum_{i=1}^pr_{f_i}m\otimes \sum_{i=1}^pr_{f_i}n)\\
& = \theta_{X}(r_fm\otimes r_fn).
\end{split}\end{equation*}
The second diagram commutes because
\begin{equation*}\begin{split}
t_f\theta_{\mathbf{b}}(m\otimes r_fn)
& = t_f\theta_{X}(\sum_{i=1}^p m_i\otimes \sum_{i=1}^pr_{f_i}n)
= t_f\theta_{X}(\sum_{i=1}^p m_i\otimes r_{f_i}n)\\
& = \sum_{i=1}^p t_f\theta_{X}(m_i\otimes r_{f_i}n)
= \sum_{i=1}^p t_{f_i}\theta_{G/K_i}(m_i\otimes r_{f_i}n)\\
& = \sum_{i=1}^p \theta_{Y}(t_{f_i}m_i\otimes n)
= \theta_{Y}(\sum_{i=1}^p t_{f_i}m_i\otimes n)\\
& = \theta_{Y}(t_f m\otimes n)
\end{split}\end{equation*}
for $m = \sum_{i=1}^p m_i$, where $m_i\in M(G/K_i)$.
The third diagram commutes similarly.

Lastly, the general case: $X = \sqcup_{i=1}^pX_i, Y = \sqcup_{i=1}^p G/H_i$, and $f$ maps $X_i$ to $G/H_i$.
Thus, $r_f$ maps $M(G/H_i)$ to $M(X_i)$ and $t_f$ maps $M(X_i)$ to $G/H_i$.
For the first diagram, take $m\in G/H_i$ and $n\in G/H_j$.
Thus, $r_f m\in M(X_i)$ and $r_f n\in M(X_j)$.
If $i\neq j$, then $\theta_{Y}(m\otimes n) = 0$ and $\theta_{X}(r_fm\otimes r_fn) = 0$.
If $i = j$, then this is the case when $Y$ is an orbit.
For the second diagram, take $m\in M(X_i)$ and $n\in G/H_j$.
Thus, $t_f m\in M(G/H_i)$ and $r_f n\in M(X_j)$.
If $i\neq j$, then $\theta_{X}(m\otimes r_fn) = 0$ and $\theta_{Y}(t_f m\otimes n) = 0$.
If $i = j$, it reduces to the case when $Y$ is an orbit.
The third diagram commutes similarly.

Given a Dress pairing $\theta\in\text{Dress}(M, N; P)$, we define $\beta$, a map from $\Phi M\square \Phi N$ to $P$ grade by grade, as follows:
\begin{equation*}\begin{split}
\beta_H:I^H_H(\Phi M\square\Phi N)&\longrightarrow I^H_H\Phi P = P(G/H)\\
M(G/H)\otimes N(G/H)\ni m\otimes n&\longmapsto \theta_{G/H}(m\otimes n)\\
I_K^HA_K\ni I^H_K(x)&\longmapsto I^H_K(\beta_K(x))
\end{split}\end{equation*}

Having constructed $\beta$, we now proceed to show that $\beta$ is indeed a map of $\mu_R(G)$-modules.
$\beta$ is linear in multiplication by elements in $\mu_R(G)$:
It is enough to check this for the generators $R, I, C$'s.
Say $K<H\leq G$ and take $m\otimes n\in M(G/H)\otimes N(G/H)$.
Then
\begin{equation*}\begin{split}
R^H_K\beta(m\otimes n) 
& = R^H_K\theta_{G/H}(m\otimes n) = \theta_{G/K}(R^H_K m\otimes R^H_K n)\\
& = \beta(R^H_K(m\otimes n))\\
C_{g, H}\beta(m\otimes n)
& = C_{g, H}\theta_{G/H}(m\otimes n) = \theta_{G/^gH}(C_{g, H}m\otimes C_{g, H}n)\\
& = \beta(C_{g, H}(m\otimes n)).
\end{split}\end{equation*}
Take $x\in A_K$ and $S\in\mu_R(G)$.
Since $SI^H_K$ acts on the grade $K$, $SI^H_K\beta(x) = \beta(SI^H_Kx)$ by induction.
Therefore, $$S\beta(I^H_Kx) = SI^H_K(x) = \beta(SI^H_Kx).$$

Let us show that $\beta$ maps the Frobenius reciprocity submodule $FR_H$ to zero for each $H\leq G$.
Take $K<H$, $a\in M(G/H)$ and $b\in N(G/K)$.
By the second commuting diagram in lemma \ref{lem:dress}, we have 
\begin{equation*}\begin{split}
\beta(I^H_K(R^H_Ka\otimes b)) 
& = I^H_K\beta(R^H_Ka\otimes b) = I^H_K\theta_{G/K}(R^H_Ka\otimes b)\\ 
& = \theta_{G/H}(a\otimes I^H_K b) = \beta(a\otimes I^H_K b).
\end{split}\end{equation*}
Take $K<H$, $c\in M(G/K)$ and $d\in N(G/H)$.
By the third commuting diagram in lemma \ref{lem:dress}, we have
\begin{equation*}\begin{split}
\beta(I^H_K(c\otimes R^H_Kd)) 
& = I^H_K\beta(c\otimes R^H_Kd)  = I^H_K\theta_{G/K}(c\otimes R^H_Kd)\\
& = \theta_{G/H}(I^H_K c\otimes d) = \beta(I^H_Kc\otimes d).
\end{split}\end{equation*}

Lastly, it is easy to see the composition of those two maps above is identity in either way.
For instance, the map $$\text{Dress}(M, N; P)\to \hom_{\B}(\Phi M\square\Phi N, \Phi P)\to \text{Dress}(M, N; P)$$, which maps $\theta\mapsto\beta\mapsto\theta'$, is identity because $\theta'_{G/H}$ equals to the restriction of $\beta$ to the first summand of the grade $H$, which in turn equals to the maps $\theta_{G/H}$ according to the constructions above.
Thus, $\theta' = \theta$.
\end{proof}

\begin{proof}[Proof for Theorem \ref{thm:main}]
Fix three Mackey functors $M,N,P\in\A$.
By Lemma \ref{lem:main_lem}, there is a natural bijection $$\text{Dress}(M, N; P)\cong\hom_{\B}(\Phi M\square\Phi N, \Phi P).$$
By Lemma \ref{lem:dress}, there is a natural bijection $$\hom_{\A}(M\square N, P)\cong\text{Dress}(M, N; P).$$
Since $\Phi:\A\xrightarrow[]{\cong}\B$ is an equivalence of categories, there is a natural bijection $$\hom_{\A}(M\square N, P)\cong\hom_{\B}(\Phi (M\square N), \Phi P).$$
Therefore, we get a natural bijection $$\hom_{\B}(\Phi M\square\Phi N, \Phi P)\cong\hom_{\B}(\Phi (M\square N), \Phi P).$$
Therefore, $\Phi M\square \Phi N$ is naturally isomorphic to $\Phi (M\square N)$ as a left $\mu_R(G)$-module.

Moreover, the equality $\Phi(B^G(-)) = \oplus_{H\leq G}B^G(G/H) = \oplus_{H\leq G}\hom_{\Omega_R(G)}(G/H, G/G)$ verifies the correspondence of units for box products in $\A$ and $\B$.
Thus, the equivalence $\Phi$ is monoidal.
\end{proof}

\end{document}